\documentclass[final,3p,times]{elsarticle}

\usepackage{hyperref}
\usepackage{amssymb}
\usepackage{amsmath,bm}
\usepackage{ntheorem}
\usepackage[titletoc]{appendix}
\usepackage{epsfig}
\usepackage{enumitem}
\usepackage{lineno}
\usepackage{latexsym}

\def\qed{\hfill$\Box$}

\newproof{proof}{Proof}
\newtheorem{theorem}{\textbf{Theorem}}
\newtheorem{corollary}[theorem]{\textbf{Corollary}}
\newtheorem{lemma}[theorem]{\textbf{Lemma}}
\newtheorem{observation}[theorem]{\textbf{Observation}}
\newtheorem{proposition}[theorem]{\textbf{Proposition}}
\newtheorem{defn}[theorem]{\textbf{Definition}}

\newtheorem{claimm}[theorem]{\textbf{Claim}}
\newtheorem*{statement}{\textbf{Statement~}}

\newproof{pot1}{Proof of Theorem \ref{thm:constup}}

\journal{Discrete Applied Mathematics}

\begin{document}
\begin{frontmatter}



\title{Bisecting and $D$-secting families for set systems}


\author[label0]{Niranjan Balachandran \fnref{ack1}}
 \ead{niranj@iitb.ac.in}
 \address[label0]{Department of Mathematics, Indian Institute of Technology, Bombay 400076, India\fnref{label3}}
 \fntext[ack1]{The research of the author is supported by grant
 	12IRCCSG016, IRCC, IIT Bombay.}
 
 \author[label1]{Rogers Mathew}
 \ead{rogers@cse.iitkgp.ernet.in}
 \author[label1]{Tapas Kumar Mishra \corref{cor1}\fnref{ack}}
 \ead{tkmishra@cse.iitkgp.ernet.in}
 \fntext[ack]{The research of the author is supported by the doctoral fellowship program of Ministry of Human Resources and Development, Govt. of India.}
 
 \author[label1]{Sudebkumar Prasant Pal}
 \ead{spp@cse.iitkgp.ernet.in}
 
 \cortext[cor1]{Corresponding author}
 \address[label1]{Department of Computer Science and Engineering, Indian Institute of Technology, Kharagpur 721302, India}


\begin{abstract}
Let $n$ be any positive integer and $\mathcal{F}$ be a family of subsets of $[n]$.
A family $\mathcal{F}'$ is said to be $D$-\emph{secting} for $\mathcal{F}$ if
for every $A \in \mathcal{F}$, there exists a subset $A' \in \mathcal{F}'$ such that $|A \cap A'| - |A \cap ([n] \setminus A')|=i$,
where $i \in D$, $D \subseteq \{-n,-n+1,\ldots,0,\ldots,n\}$. 
A $D$-\emph{secting} family $\mathcal{F}'$ of $\mathcal{F}$, where $D=\{-1,0,1\}$, is a \emph{bisecting} family  
ensuring  
the existence of a subset $A' \in \mathcal{F}'$ such that $|A \cap A'| \in \{\lceil \frac{|A|}{2}\rceil,\lfloor \frac{|A|}{2}\rfloor\}$, for each $A \in \mathcal{F}$.
In this paper, we study 
$D$-secting families for  $\mathcal{F}$ with restrictions
on $D$, and the cardinalities of $\mathcal{F}$ and the subsets of $\mathcal{F}$.
\end{abstract}


\end{frontmatter}


\section{Introduction}
\label{sec:intro}

Let $n$ be any positive integer and $\mathcal{F}$ be a family of subsets of $[n]$.
Another family $\mathcal{F'}$ of subsets of $[n]$ is called a \textit{bisecting family} for $\mathcal{F}$,
if for each subset $A \in \mathcal{F}$, there exists a subset $A' \in \mathcal{F'}$
such that $|A \cap A'| \in \{\lceil \frac{|A|}{2}\rceil,\lfloor \frac{|A|}{2}\rfloor\}$.
What is the minimum cardinality of a bisecting family  for any family $\mathcal{F}$?
We pose a more general problem based on the difference between $|A \cap A'|$ and $|A \cap ([n] \setminus A')|$.
We say a family  $\mathcal{F}'$ is $D$-\emph{secting} for  $\mathcal{F}$ if
for each subset $A \in \mathcal{F}$, there exists a subset $A' \in \mathcal{F'}$ such that $|A \cap A'| - |A \cap ([n] \setminus A')|=i$,
where $i \in D$, $D \subseteq \{-n,-n+1,\ldots,0,\ldots,n\}$.
Let $\beta_D(\mathcal{F})$ denote the minimum cardinality of a 
$D$-secting family for  $\mathcal{F}$.
In particular, when $D=\{-1,0,1\}$, the family $\mathcal{F}'$ becomes a bisecting family for $\mathcal{F}$.
We study two  cases depending on $D$: (i) $D=\{-i,-i+1,\ldots,0,\ldots,i\}$, and (ii) $D=\{i\}$, for some $i \in [n]$.
Observe that if $D=\{i\}$, only those sets $A \in \mathcal{F}$ for which
$|A|  \cong i \pmod 2$ can attain a value of $i$ for $|A \cap A'| - |A \cap ([n] \setminus A')|$. So, we consider only those 
sets for which $|A|  \cong i \pmod 2$, when $D=\{i\}$.
We define $\beta_D(n)$ as the maximum of $\beta_D(\mathcal{F})$ over all families $\mathcal{F}$ on $[n]$ and
$\beta_D(n,k)$ as the maximum of $\beta_D(\mathcal{F})$ over all families $\mathcal{F} \subseteq \binom{[n]}{k}$.
When $D=\{i\}$ ($D=\{-i,-i+1,\ldots,i\}$), we sometimes abuse the notation to denote $\beta_D(\mathcal{F})$ by $\beta_i(\mathcal{F})$ (resp., $\beta_{[\pm i]}(\mathcal{F})$).

Consider an example family $\mathcal{F}$ which consists of all the $4$-element subsets of $\{1,\ldots,6\}$.
Note that since each subset $A \in \mathcal{F}$ has an even cardinality, $\beta_0(\mathcal{F})=\beta_{[\pm 1]}(\mathcal{F})$.
Let $\mathcal{F}'=\{\{1,2,3\},\{1,2,4\},\{1,3,5\}\}$. 
It is not hard to verify that 
every 4-element subset $A \in \mathcal{F}$ is bisected by at least one element in $\mathcal{F}'$.
So, $\beta_0(\mathcal{F}) \leq 3$, for $\mathcal{F}=\binom{[6]}{4}$.
In fact there is no pair of subsets of $\{1,\ldots,6\}$ such that every 4-element subset $A \in \mathcal{F}$ is bisected by 
one of them, which is asserted by Proposition \ref{prop:n2}.
Therefore, $\beta_0(\mathcal{F}) = 3$.

\subsubsection*{Discrepancy and $D$-secting families}
Bisecting families may also be interpreted in terms of `discrepancy' of hypergraphs under multiple bicolorings.
Let $G(V,E)$  be a hypergraph with vertex set $V=\{v_1,\ldots,v_n\}$ and hyperedge set $E=\{e_1,\ldots,e_m\}$.
Given a bicoloring $X$, $X : V \rightarrow \{-1,+1\}$,
let $\mathbb{C}_{X}(e)= |\sum_{v \in e} X(v)|$ denote the discrepancy of the hyperedge $e$ under the bicoloring $X$.
Then, the \emph{discrepancy} of the hypergraph $G$, denoted by  $disc(G)$, is defined as
$disc(G)= \min_{X} \max_{e \in E} \mathbb{C}_X(e)$.
For definitions, results, and extensions of discrepancy and related problems,
see \cite{chaz2000,mat1999,chen2014,beck1996}.
Below, we define $\beta_{D}(E)$ in terms of the discrepancy of a hypergraph $G(V,E)$, where $D=[\pm i]$.
Let $t \in \mathbb{N}$ be the minimum number such that there exists a set of $t$ hypergraphs 
$G_1,\ldots, G_t$ on vertex set $V=[n]$ with
(i) $disc(G_j) \in {[\pm i]} $, for $1\leq j \leq t$, and,
(ii) $\cup_{j=1}^t G_j = G(V,E)$.
Given an optimal $D$-secting family $\mathcal{F}'$ of $E$, it is easy to construct a set of hypergraphs 
$G_1,\ldots, G_{|\mathcal{F}'|}$ satisfying the above conditions.
Again, given a set of $t$ hypergraphs 
$G_1,\ldots, G_t$ satisfying conditions (i) and (ii) under bicolorings $X_1,\ldots,X_t$, respectively,
let $(A_j^{+1},A_j^{-1})$ be the bipartition of $V$ formed by the bicoloring $X_j$.
Then, $\mathcal{F'}=\{A_1^{+1},\ldots,A_t^{+1}\}$  is a $D$-secting family for $E$.
Thus, $\beta_{[\pm i]}(E)=t$.
Moreover, the discrepancy of a hypergraph $G([n],E)$ can be defined in terms of  $\beta_{[\pm i]}(E)$ as follows.
The \emph{discrepancy} of a hypergraph $G([n],E)$ is the minimum $i \in \mathbb{N}$ such that $\beta_{[\pm i]}(E)=1$.

\subsubsection*{Separating and bisecting families}

Given a family $\mathcal{F}$ of subsets of $[n]$, finding another family $\mathcal{F}'$ with certain properties has been well investigated.
One of the most studied problem in this direction is the computation of \emph{separating families}.
Let $\mathcal{F}$ consist of pairs $\{i,j\}$, $i,j \in \mathbb{N}$, $i \neq j$ and $\mathcal{F}'=\{A_1',\ldots,A_t'\}$
be another family of subsets on $[n]$ ($\mathcal{F}$ can be viewed as the edge set of a graph on vertex set $[n]$).
A subset $A_l'$ separates a pair $\{i,j\}$ if $i \in A_l'$ and $j \not\in A_l'$ or vice versa, $l \in [t]$.
The family $\mathcal{F}'$ is a separating family for $\mathcal{F}$ if every pair $\{i,j\} \in \mathcal{F}$ is separated by some  $A' \in \mathcal{F}'$. 
It is easy to see that $\mathcal{F}'$ is indeed a bisecting family for $\mathcal{F}$.
Let $f(n)$ denote the size of a minimum separating family $\mathcal{F}'$ for a family $\mathcal{F}$
consisting of all the $\binom{n}{2}$ pairs (edge set of a complete graph on $n$ vertices).
R\'{e}nyi \cite{renyi1961} proved that $f(n)=\lceil \log_2 n\rceil$.
Observe that $f(n)$ is the minimum  number of bipartite graphs needed to cover the edges of a complete graph $K_n$.
We note the following generalization of the above statement for arbitrary graphs.
\begin{proposition}[Folklore]
	\label{lemma:bc}
	Let $\chi(G)$ denote the chromatic number of graph $G$.
	Then, $\lceil \log_2 \chi(G) \rceil$
	bipartite graphs are necessary and sufficient to cover the edges of $G$. 
\end{proposition}

Note that $f(n)$ is equal to $\beta_0(n,2)$, thus $\beta_0(n,2)=\lceil \log_2 n\rceil$.
In fact, when the family $\mathcal{F}$ is the edge set of a graph $G(V,E)$, where $V=[n]$,
any bisecting family $\mathcal{F}'$ for  $\mathcal{F}$ forms a covering  of the edges of $G$ with $|\mathcal{F}'|$ bipartite graphs. 
We state these observations as a corollary below. 
\begin{corollary}\label{prop:graph}
	For a graph $G(V,E)$,
	$\beta_{0}(E)=\lceil \log_2 \chi(G) \rceil$.
	Thus, $\beta_0(n,2)=\lceil \log_2 n\rceil$.
\end{corollary} 
See \cite{renyi1961,kat1966,weg1979} for details on separating families.

Galvin proposed the following special case restricted only subsets of size exactly half the size of the ground set:
What is the minimum $m$  such that there exists a set of 
subsets $B_1,\ldots,B_m$ of $\{1,\ldots,4n\}$, each of size $2n$, with the property that for all $A \subset \{1,\ldots,4n\}$, there exists an $B_i$ with $|A \cap B_i|=n$.
He showed that $m \leq 2n$ and conjectured that $m=2n$.
Frankl-Rodl \cite{frankl1987} proved that $m > \epsilon n$, for some fixed $\epsilon$, $0< \epsilon < 1$. Enamoto et. al.
\cite{enomoto1987} demonstrated that $m=2n$ when $n$ is odd.

\subsection{Notations and definitions}

Let $[n]$ denote the set of integers $\{1,\ldots,n\}$,
$\pm i$ denote the set of integers $\{-i,i\}$, and
$[\pm i]$ denote the set of integers $\{-i,-i+1,\ldots,i\}$.
Let $\mathcal{F}$ denote a family of subsets of $[n]$ and $\mathcal{F}'$ denote another family of subsets
with some desired intersection property with elements of $\mathcal{F}$.
Let $\binom{[n]}{k}$ denote the family of all the $k$-sized subsets of $[n]$.
We use $\beta_{[\pm i]}(\mathcal{F})$ (resp., $\beta_i(\mathcal{F})$) to denote $\beta_D(\mathcal{F})$ if $D=[\pm i]$ (resp., $D=\{i\}$).
We denote an $n$-dimensional vector $R \in \{0,1\}^n$ (or $\{-1,+1\}^n$) as $R=(x_1,\ldots,x_n)$ where $x_j \in \{0,1\}$ (resp., $\{-1,+1\}$).
The \emph{weight} of a vector  $R=(x_1,\ldots,x_n)\in \{0,1\}^n$ (or $\{-1,+1\}^n$) is 
the number of $x_j$'s which are 1 (resp., -1), $1 \leq j \leq n$.
Vector $R \in \{0,1\}^n$ is even (resp., odd) if the number of $1$'s in $R$ is even (resp., odd).
A vector $R \in \{-1,1\}^n$ is even (resp., odd) if the number of $-1$'s in $R$ is even (resp., odd).
We use $\log$ to denote $\log_2$ in the rest of the paper.

\subsection{Our Contribution}

We begin by addressing  the problem of bounding and computing $\beta_{D}(n)$, where $D=[\pm i]$.
We demonstrate a  construction yielding an upper bound of $\lceil \frac{n}{2i} \rceil$ for $\beta_{[\pm i]}(n)$. 
Further,
we show using a polynomial representation for the parity function
that $\lceil \frac{n}{2i} \rceil$ is also a lower bound for $\beta_{[\pm i]}(n)$.

\begin{theorem}\label{thm:Knn}
	$\beta_{[\pm i]}(n) = \lceil \frac{n}{2i} \rceil$, $n \in \mathbb{N}$, $i \in [n]$.
\end{theorem}

We study $\beta_{[\pm i]}(\mathcal{F})$ for a family $\mathcal{F}$ on $[n]$,
in terms of $i$ and $|\mathcal{F}|$, using Chernoff's bound.
\begin{theorem}\label{thm:prob}
	Let $\mathcal{F}$ be a family of subsets of $[n]$ and let $m=|\mathcal{F}|$.
	Let $D=[\pm i]$, where $i \geq \sqrt{\frac{3n\ln (2m)}{t}}$ and $t \leq \frac{1}{2}\log m$.
	Then, $\beta_D(\mathcal{F}) \leq t$. 
\end{theorem}
In particular, if $i \geq \sqrt{4.2n+1}$ and $|\mathcal{F}| = O(n^c)$, for $c \in \mathbb{N}$, a $D$-secting family $\mathcal{F}'$ of cardinality $O(\log n)$ can be computed for families $\mathcal{F}$, thus improving the bound from Theorem \ref{thm:Knn}
for this range of $i$ and $|\mathcal{F}|$.

Subsequently, we study $\beta_{D}(n)$, where $D$ is a singleton set, i.e., $D=\{i\}$.
Note that $\beta_{i}(n)=\beta_{-i}(n)$.
Moreover, when $D=\{-i,i\}$,  
note that $\beta_{\pm i}(n)\leq \beta_{i}(n) \leq 2\beta_{\pm i}(n)$.
Therefore, we focus on establishing bounds for $\beta_{i}(n)$.
We demonstrate a construction to show that  $\beta_1(n)$ is at most $\lceil\frac{n}{2}\rceil$.
We also show that $\beta_1(n)$ is at least $\lceil \frac{n}{2} \rceil$ using arguments similar to those in the proof of Theorem 
\ref{thm:Knn} about $\beta_{[\pm 1]}(n)$.
In Section \ref{subsec:idist}, we establish a lower bound of $\frac{n-i+1}{2}$ for arbitrary $i \in [n]$, $i \geq 2$.
We demonstrate a construction establishing $\beta_i(n) \leq n-i+1$.
We have the following theorem.
\begin{theorem}\label{thm:iknn}
	$\frac{n-i+1}{2} \leq \beta_i(n) \leq n-i+1$, $n \in \mathbb{N}$, $i \in [n]$.
\end{theorem}

In Section \ref{sec:k-uniform}, we consider families $\mathcal{F}$, $\mathcal{F} \subseteq \binom{[n]}{k}$.
We study $\beta_{[\pm 1]}(n,k)$ in detail when $k$ is even;
the analysis for $\beta_{i}(n,k)$ for $i \in [n]$ and for the case when $k$ is odd is analogous.
We have lower bounds for $\beta_{[\pm 1]}(n,k)$ given by Theorem \ref{thm:klower},
Observation \ref{obs:ineq:2} (see Section \ref{sec:prelim}), and Theorem \ref{thm:n_k_linear} which are useful when
$k$ is a constant, $k$ is sublinear in $n$, and $k$ is linear in $n$, respectively.
We establish the following theorem using entropy based arguments. 
\begin{theorem}\label{thm:klower}
	\begin{align*}
	\beta_{[\pm 1]}(n,k) & \geq \begin{cases}
	\log (n-k+2) \text{, when $k$ is even and $\frac{k}{2}$  is odd,}\\
	\lceil (\log \lceil \frac{n}{\lceil\frac{k}{2}\rceil} \rceil)\rceil \text{, for any $k \geq 2$}.
	\end{cases}
	\end{align*}
\end{theorem}
When $cn < k < (1-c)n$ for a constant $c$, $0 < c < \frac{1}{2}$, we establish an improved lower bound for $\beta_{[\pm 1]}(n,k)$ 
using a vector space orthogonality argument, enabling us to apply a recent result of Keevash and Long \cite{keevash2017frankl}.
\begin{theorem}\label{thm:n_k_linear}
	Let $c$ be a constant such that $0 < c < \frac{1}{2}$ and $n \in \mathbb{N}$. If $cn < k < (1-c)n$, then
	\begin{align*}
	\max\Big\{\beta_{[\pm 1]}(n,k),\beta_{[\pm 1]}(n,k-1),\beta_{[\pm 1]}(n,k-2),\beta_{[\pm 1]}(n,k-3)\Big\}	\geq \delta n, 	
	\end{align*}
	where $\delta=\delta(c)$ is some real positive constant.
\end{theorem}

Let $\mathcal{F}$ be a family of subsets of $[n]$.
The {\it dependency} of a subset $A \in \mathcal{F}$
denoted by $d(A,\mathcal{F})$ 
is the number of subsets $\widehat{A} \in \mathcal{F}$, such that
(i) $|A \cap \widehat{A}| \geq 1$, and
(ii) $A \neq \widehat{A}$.   
The {\it dependency of a family} $d(\mathcal{F})$ or simply $d$, denotes the 
maximum dependency of any subset $A$ in the family $\mathcal{F}$.
We study  $\beta_{[\pm 1]}(\mathcal{F})$ for families $\mathcal{F}$ consisting of $k$-sized sets with bounded dependency and
using a corollary of the Lov\'{a}sz local lemma from \cite{Moser:2010:CPG:1667053.1667060}, we prove the following probabilistic upper bound.
\begin{theorem}\label{thm:lll}
	For a family $\mathcal{F}$ consisting of $k$-sized subsets of $[n]$ and dependency $d$,
	$\beta_{[\pm 1]}(\mathcal{F}) \leq \frac{\sqrt{k}}{c}(\ln(d+1)+1)$, where $c = 0.67$.
\end{theorem}

We also study the case when $\mathcal{F}$ consists of all the subsets of $[n]$ of cardinality more than $k$, $k \in [n]$
and we have the following bounds. 
\begin{theorem}\label{thm:constup}
	Let $\mathcal{F}=\binom{[n]}{k} \cup \binom{[n]}{k+1}\ldots \cup \binom{[n]}{n}$.
	Then, $ \frac{n-k+1}{2} \leq \beta_{[\pm 1]}(\mathcal{F}) \leq \min \{\frac{n}{2}, n-k+1 \}.$
\end{theorem}
Note that when $n-k$ is a constant, Theorem \ref{thm:constup} gives better upper bounds for $\beta_{[\pm 1]}(\mathcal{F})$.


\subsection{Some quick observations}
\label{sec:prelim}
In this section, we derive a few basic results on $\beta_D(\mathcal{F})$, $\beta_D(n)$ and $\beta_D(n,k)$.
$\mathcal{P}$ is a \emph{property} for a set system if it is invariant under isomorphism\footnote{
Two set systems $H=(X;E_1,E_2,\ldots,E_m)$ and $I = (Y;F_1,F_2,\ldots,F_m)$
are said to be isomorphic if they have the same number $m$ of subsets, and if there
exists a bijection $\varphi: X \rightarrow Y$ and a permutation $\pi$  on $M= \{1,2,\ldots, m\}$ such
that \[\varphi(E_i) = F_{\pi(i)} ~~ (i = 1, 2, ..., m) .\] 	
See page 411 of \cite{berge1973graphs} for related notions.}.
It is not hard to see that for any two isomorphic families $\mathcal{F}_1$ and $\mathcal{F}_2$ on $[n]$, $\beta_D(\mathcal{F}_1)=\beta_D(\mathcal{F}_2)$.
So, $\beta_D$ is a property of the set system.
For any two families $\mathcal{F}_1$ and $\mathcal{F}_2$, $\mathcal{F}_1 \subseteq \mathcal{F}_2$, $\beta_D(\mathcal{F}_1) \leq \beta_D(\mathcal{F}_2)$.
Therefore, $\beta_D(n)$ and $\beta_D(n,k)$ are monotone with respect to $n$.
However, $\beta_D(n,k)$ is not monotone with respect to  $k$: $\beta_{[\pm 1]}(n,2) = \lceil \log n\rceil$ (see Corollary \ref{prop:graph}),
$\beta_{[\pm 1]}(n,\frac{n}{2}) = \Omega(\sqrt{n})$ (From Observation  \ref{obs:ineq:2}) whereas 
$\beta_{[\pm 1]}(n,n-2) = 3$ (see Proposition \ref{prop:n2}).

We note that for any integer $t$, ``$\beta_D(\mathcal{F}) \leq t$'' is not {\it hereditary}\footnote{For a family $\mathcal{F}=\{A_1,\ldots,A_m\}$ on $[n]$, and a set $S \subseteq [n]$, the family $\mathcal{F}_S=\{A^s_1,\ldots,A^s_m\}$ 
	is called a family induced by $S$ on $\mathcal{F}$ if $A^s_j= A_j \cap S$, for $1\leq j \leq m$.
	A property $\mathcal{P}$ is {\it hereditary} if  $\mathcal{F} \in \mathcal{P}$ implies  $\mathcal{F}_S \in \mathcal{P}$ for every induced family $\mathcal{F}_S$ of $\mathcal{F}$, $S \subseteq [n]$.}.
This can be demonstrated with the following example.
Let $\mathcal{F}=\{\{1,2,4,5\},\{1,3,4,5\},\{2,3,4,5\}\}$ be a family on $\{1,\ldots,5\}$ and $S=\{1,2,3\}$.
$\mathcal{F}_S=\{\{1,2\},\{1,3\},\{2,3\}\}$ is the subfamily of $\mathcal{F}$ induced by $S$.
It is easy to see that when $D=[\pm 1]$, $\beta_D(\mathcal{F})=1$ whereas $\beta_D(\mathcal{F}_S)=2$.

\begin{observation}\label{obs:1}
	Let $\mathcal{F}$ be a family of subsets of $[n]$ and
	$\mathcal{F'}=\{S_1,\ldots,S_r\}$ be a $D$-secting family for $\mathcal{F}$, $r \in \mathbb{N}$ and $D=[\pm i]$.
	Then, 	$\mathcal{H}=\{H_1,\ldots,H_r\}$ is also a $D$-secting family for $\mathcal{F}$,
	where $H_i \in  \{[n]\setminus S_i, S_i\}$, $1\leq i \leq r$.
\end{observation}

For the rest of the section, assume that $n$ is even (since it does not effect the asymptotics).
Note that when $k$ is even (resp., odd), the maximum number of $k$-sized sets $A \in \mathcal{F}$ 
that can be bisected with any set $A' \subseteq [n]$ is $\binom{\frac{n}{2}}{\frac{k}{2}}^2$ 
(resp., $2\binom{\frac{n}{2}}{\lceil\frac{k}{2}\rceil}\binom{\frac{n}{2}}{\lfloor\frac{k}{2}\rfloor}$), $k \in [n]$.
This gives a trivial lower bound for $\beta_{[\pm 1]}(n,k)$ using Stirling's approximation, i.e., $\sqrt{2\pi n}(\frac{n}{e})^n \leq n! \leq e\sqrt{n}(\frac{n}{e})^n$.
\begin{observation}\label{obs:ineq:2}
	\begin{align}
	\beta_{[\pm 1]}(n,k) \geq \frac{\binom{n}{k}}{2\binom{\frac{n}{2}}{\lceil\frac{k}{2}\rceil}\binom{\frac{n}{2}}{\lfloor\frac{k}{2}\rfloor}}= \Omega( \sqrt{\frac{k (n-k)}{n}}).
	\end{align}
\end{observation}

The constant in the lower bound is $C=\frac{\sqrt{2}\pi^{2.5}}{e^4} \geq .45$.
When $k=\frac{n}{2}$, this corresponds to a lower bound of $\Omega(\sqrt{n})$ for $\beta_{[\pm 1]}(n,\frac{n}{2})$.
Moreover, using the monotone property, $\beta_{[\pm 1]}(n) \geq \beta_{[\pm 1]}(n,\frac{n}{2})=\Omega(\sqrt{n})$. 
In what follows, we derive improved upper bounds and lower bounds for  $\beta_D(n)$.
We start our discussion with the case  $D=[\pm i]$, $i \in [n]$, followed by the case $D=\{i\}$.

\section{Bounds for $\beta_{[\pm i]}(n)$}
\label{sec:Knn}

Recall that $\beta_{[\pm i]}(n)$ is the maximum of $\beta_{[\pm i]}(\mathcal{F})$ 
over all families $\mathcal{F}$ on $[n]$,
where $\beta_{[\pm i]}(\mathcal{F})$ denotes the minimum cardinality of a 
$[\pm i]$-secting family for  $\mathcal{F}$.

\subsection{Upper bounds}
\label{subsec:upknn}

\begin{lemma}\label{lemma:Knn1}
	$\beta_{[\pm i]}(n) \leq \lceil\frac{n}{2i}\rceil$.
\end{lemma}

\begin{proof}
	Let $\mathcal{F}$ denotes the family consisting of all the non-empty subsets of $[n]$.
	In what follows, we demonstrate a construction that yields a $[\pm i]$-secting family of cardinality $\frac{n}{2i}$ 
	for $\mathcal{F}$, assuming  $2i$ divides $n$.
	Let $B_1=\{1,2,\ldots,\frac{n}{2}\}$.
	The set $B_2$ is obtained from $B_1$ by swapping the largest $i$ elements of $B_1$ with the smallest $i$ elements in $[n]\setminus B_1$. So,  $B_2=\{1,2,\ldots,\frac{n}{2}-i,\frac{n}{2}+i,\frac{n}{2}+i-1,\ldots,\frac{n}{2}+1\}$ (we write the swapped elements in descending order for convenience).
	In general, $B_{j+1}$ is obtained from $B_{j}$ by swapping the 
	largest $i$ elements of $B_1 \cap B_j$ (i.e., $\{\frac{n}{2}-ij+1,\ldots,\frac{n}{2}-ij+i \}$)
	with the smallest $i$ elements of $([n] \setminus B_1) \cap ([n] \setminus B_j)$ (i.e., $\{\frac{n}{2}+ij-i+1,\ldots,\frac{n}{2}+ij\}$). 
	We stop the process at $B_{\frac{n}{2i}}=\{1,\ldots,i,n-i,n-(i-1),\ldots,\frac{n}{2}+1\}$. 
	Let $\mathcal{F}'=\{B_1,\ldots,B_\frac{n}{2i}\} $.
	
	We prove that $\mathcal{F}'$ is indeed a  $[\pm i]$-secting family for $\mathcal{F}$.
	For the sake of contradiction, we assume that there exists some $A \subseteq [n]$  such that 
	$|A \cap B_j| -  |A \cap ([n] \setminus B_j)| \not\in D$, for all $B_j \in \mathcal{F}'$.
	Let $c_j {\mathrel{\mathop:}=} |A \cap B_j| -  |A \cap ([n] \setminus B_j)|$, $1 \leq j \leq {\frac{n}{2i}}$ . 
	From the construction of $B_{j+1}$ from $B_j$, observe that  $|c_j-c_{j+1}| \leq |B_j \triangle B_{j+1}|=2i$, $1 \leq j \leq {\frac{n}{2i}}-1$.
	Clearly, $c_1=d$, for some $d \not\in \{-i,\ldots,i\}$.
	
	\begin{claimm}\label{claim:1}
		$c_{\frac{n}{2i}} \leq  -d+2i$ for $d > 0$ (resp. $\geq -d-2i$ for $d < 0$).
	\end{claimm}
	\begin{proof}
		Let $B_{\frac{n}{2i}+1}$ be the set obtained from $B_{\frac{n}{2i}}$ 
		by swapping the largest $i$ elements $\{ 1,\ldots,i \}$ of $B_1 \cap B_{\frac{n}{2i}}$
		with the smallest $i$ elements $\{n-i+1,\ldots,n\}$ of $([n] \setminus B_1) \cap ([n] \setminus B_{\frac{n}{2i}})$.
		Let $c_{\frac{n}{2i}+1}=|A \cap B_{\frac{n}{2i}+1}| -  |A \cap ([n] \setminus B_{\frac{n}{2i}+1})|$.
		Observe that since $c_1=d$ and $B_{\frac{n}{2i}+1}$ is $[n] \setminus B_1$, $c_{\frac{n}{2i}+1}=-d$.
		Moreover, $|c_{\frac{n}{2i}}-c_{\frac{n}{2i}+1}| \leq 2i$.
		So, $c_{\frac{n}{2i}}$ is at most  $-d+2i$.
		The proof for the case of $d < 0$ is similar.
		\qed
	\end{proof}
	
	We now have these exhaustive cases.
	\begin{enumerate}
		\item $d \geq 2i$ (or $d \leq -2i$): 
		Note that $D =\{-i,\dots,+i\}$ and  $|c_j-c_{j+1}| \leq 2i$, for all $1 \leq j \leq {\frac{n}{2i}}-1$.
		Using Claim \ref{claim:1}, $c_{\frac{n}{2i}} \leq 0$ (resp., $c_{\frac{n}{2i}} \geq 0$). 
		Therefore, there exists at least one index $l$, $1 \leq l \leq {\frac{n}{2i}}-1$,
		such that $c_l \cdot c_{l+1} \leq 0$. Observe that either of $c_l$ or $c_{l+1}$, or both lie in  $\{-i,\dots,+i\}$.
		This is a contradiction to our assumption that $A$ is not $D$-sected by $\mathcal{F}'$. 
		\item $i < d < 2i$: From Claim \ref{claim:1}, it is clear that $c_{\frac{n}{2i}} < i$.
		So, if there exists an index $l$, $1 \leq l \leq {\frac{n}{2i}}-1$,
		such that $c_l \cdot c_{l+1} \leq 0$, either $c_l$ or $c_{l+1}$ or both lie in  $\{-i,\dots,+i\}$.
		Otherwise, $c_{\frac{n}{2i}} \in \{0,\dots,i-1\} \subset D$ as desired.
		\item $-2i < d < -i$: Similar to the previous case.
	\end{enumerate}
	
	This establishes that $\beta_{[\pm i]}(n)$ is at most $\frac{n}{2i}$, when $2i$ divides $n$.
	Note that when $n$ is not divisible by $2i$, we can construct $\mathcal{F}'$ of cardinality  $\lceil\frac{n}{2i}\rceil$ 
	with the same procedure, where $B_{\lceil\frac{n}{2i}\rceil}= \{1,\dots,p,n-p,n-(p-1),\dots,\frac{n}{2}+1\}$, $p=n \mod 2i$.
	This completes the proof of Lemma \ref{lemma:Knn1}.
	\qed
\end{proof}

\subsection{Lower bounds}
\label{subsec:lowknn}

To obtain a lower bound for $\beta_D(n)$, it is natural to remove 1 or 2 points  from $[n]$
and to proceed with induction.
However, we note that, even when $D=\{-1,0,1\}$, such a direct induction only yields a lower bound of $\log n$, which is 
not useful (since we already have a lower bound of $\Omega(\sqrt{n})$ from Section \ref{sec:prelim}).
In order to derive a tight lower bound 
for $\beta_D(n)$, we use  the vector representations of sets and a polynomial representation
of Boolean functions.

For any subset $A \subseteq [n]$, let
(i) $X_A=(x_1,\ldots,x_n) \in \{0,1\}^n$ be the incidence vector such that $x_i=1$ if and only if $i \in A$; and,
(ii)$R_A=(r_1,\ldots,r_n) \in \{-1,1\}^n$ be the incidence vector such that $r_i=1$ if and only if $i \in A$.
Observe that for any two subsets $A$ and $A'$ of $[n]$,
the dot product of $X_A=(x_1,\ldots,x_n)$ with $R_{A'}=(r_1,\ldots,r_n)$,
denoted by $\left\langle X_A,R_{A'}\right\rangle$, is 
equivalent to $|A \cap A'| -  |A \cap ([n] \setminus A')|$.
For an even (resp., odd) cardinality subset $A \in \mathcal{F}$, note that the corresponding incidence vector $X_A=(x_1,\ldots,x_n)$ is even (resp., odd).
Let $\mathcal{F}$ be a family of subsets of $[n]$.
Observe that for any even subset $A_e \in \mathcal{F}$ and any arbitrary subset $A' \subseteq [n]$,
$\left\langle X_{A_e},R_{A'} \right\rangle \equiv 0 \mod 2$, i.e., $\left\langle X_{A_e},R_{A'}\right\rangle \in \{0,\pm 2,\pm 4,\ldots\}$.
Moreover, for any odd subset $A_o \in \mathcal{F}$,
$\left\langle X_{A_o},R_{A'} \right\rangle \equiv 1 \mod 2$, i.e., $\left\langle X_{A_o},R_{A'}\right\rangle \in \{\pm 1,\pm 3,\pm 5,\ldots\}$.

We demonstrate that the polynomial
representation of Boolean functions \cite{saks1990,saks1993} is useful to establish lower bounds for $\beta_D(n)$.
Let $f: \{-1,1\}^n \rightarrow \{-1,1\}$ be a Boolean function on $n$ variables,say $y_1,\ldots,y_n$.
For instance, the \emph{parity}
function on $n$ variables is simply equal to the monomial $\prod_{j=1}^{n}y_j$. 
Let $sign: \mathbb{R} \setminus \{0\} \rightarrow \{0,1\}$ be a function defined as
(i) $sign(\alpha)= 1$ if $\alpha >0$, and (ii) $sign(\alpha)= 0$, otherwise,
for $\alpha \in \mathbb{R} \setminus \{0\}$.
A multilinear polynomial $P(y_1,\dots,y_n)$ \emph{weakly represents} $f$ if $P$ is nonzero and for every $Y=(y_1,\dots,y_n)$ where $P(Y)$ is nonzero, $sign(f(Y))$ = $sign(P(Y))$. 
The \emph{weak degree} of a function $f$ is
the degree of the lowest degree polynomial which weakly represents $f$. We have the following result 
that follows from Lemma 2.29 of  \cite{saks1993} originally proved by Minsky and Papert in \cite{minsky1969perceptron}.

\begin{lemma}\label{lemma:weak}
	The weak degree of the parity function on $n$ variables is $n$.
\end{lemma}

In what follows, we use the notion of weak degree of the parity function to establish Theorem \ref{thm:Knn}.

\begin{lemma}\label{lemma:Knn2}
	$\beta_{[\pm i]}(n) \geq \lceil\frac{n}{2i}\rceil$.
\end{lemma}

\begin{proof}
	
	Let $\mathcal{F}$ denote the $2^n-1$ non-empty subsets of $[n]$.
	Let $\mathcal{F}'$ be a minimum cardinality $[\pm i]$-secting family for $\mathcal{F}$.
	Let $\mathcal{R}$ be set of incidence vectors of sets in $\mathcal{F}'$, where each vector $R$ in $\mathcal{R}$ is
	an element of $\{-1,+1\}^n$.
	We start the analysis assuming $i$ is even and $i>0$, and then extend to odd $i$. 
	For every odd set $A_o \in \mathcal{F}$, there exists a vector $R \in \mathcal{R}$ such that
	$\left\langle X_{A_o}, R \right\rangle-d =0$, for some  $d \in \{-i+1,-i+3,\ldots,i-1\}$.
	Let $X=(x_1,\ldots,x_n) \in \{0,1\}^n$. We use
	$X$ to denote the incidence vector of any arbitrary set in $\mathcal{F}$.
	Consider the polynomial $M$ on $X=(x_1,\ldots,x_n)$ as
	\begin{align}\label{eq:poly1}
	M(X)= \left(\prod_{R \in \mathcal{R}} \left(\left(\left\langle X,R \right\rangle \right)^2-1^2\right) 
	\prod_{R \in \mathcal{R}} \left(\left(\left\langle X,R \right\rangle \right)^2-3^2\right) \ldots
	\prod_{R \in \mathcal{R}} \left(\left(\left\langle X,R \right\rangle \right)^2-(i-1)^2\right) \right)^2.
	\end{align}
	
	From the definitions of $\mathcal{R}$ and $M$, it is clear that $M(X)$ is 
	(i) zero when $X=X_{A_o}$ for all odd subsets $A_o \in \mathcal{F}$; and 
	(ii) positive when $X=X_{A_e}$ for all even subsets $A_e \in \mathcal{F}$.

	\subsubsection*{Domain conversion and multilinearization}
	Recall that a vector $T \in \{0,1\}^n$ is even if the number of $1$'s in $T$ is even and
	a vector $T \in \{-1,1\}^n$ is even if the number of $-1$'s in $T$ is even.
	Consider the polynomial $N$ on $Y=(y_1,\ldots,y_n)$, where each $y_i=\pm 1$.
	\begin{align}\label{eq:poly2}
	N(y_1,\ldots,y_n)=M(x_1,\ldots,x_n),
	\end{align}
	where $x_j=\frac{1-y_j}{2}$, $1 \leq j \leq n$.
	Note that if $y_i = -1$ (resp. 1), then $\frac{1-y_i}{2}$ becomes 1 (resp. 0).
	So, if some vector $Y=(y_1,\ldots,y_n)$ includes an even number of $-1$'s,
	then the vector $(\frac{1-y_1}{2},\ldots,\frac{1-y_n}{2})$ has an even number of $1$'s, i.e., the reduction of the vector $(y_1,\ldots,y_n)$ from the $\{-1,1\}^n$ domain to
	$(\frac{1-y_1}{2},\ldots,\frac{1-y_n}{2})$ in the $\{0, 1\}^n$ domain preserves the definition of {\it evenness}.
	Note that (i) $N(Y)$ evaluates to zero, when $Y=Y_{A_o} \in \{-1,1\}^n$ for all odd subsets $A_o \in \mathcal{F}$;
	(ii) $sign(N(Y)= sign(parity(Y))$, when $Y=Y_{A_e} \in \{-1,1\}^n$ for all even subsets $A_e \in \mathcal{F}$. 
	Let $N'(Y=(y_1,\ldots,y_n))$ be the multilinear polynomial obtained from $N(Y=(y_1,\ldots,y_n))$ by 
	repeatedly replacing each $y_i^2$ in the monomials by 1.
	$deg(N'(Y)) \leq deg(N(Y))$ and $N'(Y)=N(Y)$, for vectors $Y \in \{-1,1\}^n$.
	
	Clearly, $N'(Y)$ weakly represents the parity function.
	Each term $(\prod_{R \in \mathcal{R}} ((\left\langle X,R \right\rangle)^2-j^2))^2$, $j \in \{1,\ldots,(i-1)\}$, contributes a degree of $4|\mathcal{R}|$ to the  degree of $M(X)$, and, there are $\frac{i}{2}$ such terms.
	Therefore, the degree of $M(X)$ is $2|\mathcal{R}|i$.
	Moreover, from Equation \ref{eq:poly2}, $deg(N'(Y)) \leq deg(N(Y))=deg(M(X))$.
	However, from Lemma \ref{lemma:weak},
	$deg(N'(Y)) \geq n$, which implies $\beta_{[\pm i]}(n) = |\mathcal{R}| \geq  \frac{n}{2i}$.

	If $i> 1$ is odd, $M(X)$ is defined as
	\begin{align*}
	\prod_{R \in \mathcal{R}} \left((\left\langle X,R \right\rangle)^2\right)
	\left(\prod_{R \in \mathcal{R}} \left((\left\langle X,R \right\rangle)^2-2^2\right) \prod_{R \in \mathcal{R}} 
	\left((\left\langle X,R \right\rangle)^2-4^2\right)\ldots \prod_{R \in \mathcal{R}} \left((\left\langle X,R \right\rangle)^2-(i-1)^2 \right)\right)^2.
	\end{align*}
	Observe that $M(X)$ vanishes for all even vectors and is positive for all odd vectors.  
	The polynomial $N$ on $Y=(y_1,\ldots,y_n)$, where each $y_i=\pm 1$, is now defined as
	\begin{align}\label{eq:polyy2}
	N(y_1,\ldots,y_n)=-M(x_1,\ldots,x_n).
	\end{align}
	Note that degree of $M(X)$ is $2|\mathcal{R}|+4|\mathcal{R}|\frac{i-1}{2}=2|\mathcal{R}|i$ and the rest of the arguments
	are same as the previous case.
	
	We are only left with the cases when $i=0$ and $i=1$. 
	Observe that $\beta_D(n)$ for the case of $D=\{0\}$ and $D=\{-1,0,1\}$ is same:
	any bisecting family for a family $\mathcal{F}_1$ consisting of only the $2^{n-1}-1$ non-empty even subsets of $[n]$
	must bisect all the $2^{n}-1$ subsets of $[n]$.
	In this case, take $M(X)=\prod_{R \in \mathcal{R}} \left((\left\langle X,R \right\rangle)^2\right)$ and proceed 
	as before to get $\beta_{[\pm 1]}(n) \geq \frac{n}{2}$.
	
	\qed
\end{proof} 

From Lemmas \ref{lemma:Knn1} and \ref{lemma:Knn2}, Theorem \ref{thm:Knn} follows, which is restated below.

\begin{statement}
	$\beta_{[\pm i]}(n) = \lceil \frac{n}{2i} \rceil$, $n \in \mathbb{N}$, $i \in [n]$.
\end{statement}

Let $\mathcal{F}$ consists of $2^n-1$ non-empty subsets of $[n]$.
Then, Theorem \ref{thm:Knn} asserts that the construction of $[\pm i]$-secting family  of cardinality $\lceil \frac{n}{2i}\rceil$
in Section \ref{subsec:upknn} is indeed optimal.
Moreover, Theorem \ref{thm:Knn} implies that if we allow the imbalances of intersections up to $\sqrt{n}$, i.e., $D= [\pm \sqrt{n}]$,
then a family $\mathcal{F}'$ of cardinality $\frac{\sqrt n}{2}$ is necessary and sufficient for $\mathcal{F}$.
\begin{corollary}
	\label{cor:sqrtn}
	For $D=[\pm \sqrt{n}]$, $n \in \mathbb{N}$,
	$\beta_D(n) = \lceil \frac{\sqrt{n}}{2}\rceil$.
\end{corollary}
In what follows, we demonstrate that $D$-secting families of cardinality much smaller than $\frac{\sqrt n}{2}$
can be computed when $|\mathcal{F}|$ is small.

\subsection{Computing $\beta_{[\pm i]}(\mathcal{F})$ for arbitrary families}
\label{sec:chernoff}

In Section \ref{sec:intro}, we discussed about the discrepancy interpretation of the bisection problems.
Probabilistic method is an useful tool in computing low discrepancy colorings. 
The following Chernoff's bound is used extensively to establish upper bounds on the discrepancy of hypergraphs.
\begin{lemma}\label{lemma:chernoff}\cite{chaz2000}
	If $X=\sum_{i=1}^n X_i$ is the sum of $n$ independent random variables distributed uniformly over $\{-1,1\}$,
	then for any $\Delta>0$,
	\begin{align*}
	P[|X| > \Delta] \leq 2e^{-\frac{\Delta^2}{2n}}.
	\end{align*}
\end{lemma}
In what follows, we obtain an upper bound on $\beta_{[\pm i]}(\mathcal{F})$, when $\mathcal{F}$ is a family of arbitrary sized subsets,
with a simple application of Lemma \ref{lemma:chernoff}.

\subsubsection*{Proof of Theorem \ref{thm:prob}}
\begin{statement}
	Let $\mathcal{F}$ be a family of subsets of $[n]$ and let $|\mathcal{F}|=m$.
	Let $D=[\pm i]$, where $i=\sqrt{\frac{3n\ln (2m)}{t}}$ and $t \leq \frac{1}{2}\log m$.
	Then, $\beta_D(\mathcal{F}) \leq t$. 
\end{statement}

\begin{proof}
	We pick a set $\mathcal{F}'$ of $t$ random subsets $\{A_1',\ldots,A_t'\}$ of $[n]$,
	where for each $j$, $1\leq j \leq t$, a point $a \in [n]$ is chosen independently and uniformly at random into $A_j'$. 
	Let $R_{A_j'}=(r_1,\ldots,r_n) \in \{-1,1\}^n$ be the incidence vector corresponding to $A_j'$: $r_i$ is 1 if and only if $i \in A_j'$.
	For any subset $A \in \mathcal{F}$, 
	$|A \cap A_j'|-|A \cap ([n]\setminus A_j')|$ can be viewed as sum of $|A|$ random variables distributed uniformly over $\{-1,1\}$.
	We say a subset  $A \in \mathcal{F}$ is bad with respect to subset $A_j' \in \mathcal{F}'$  if $||A \cap A_j'|-|A \cap ([n]\setminus A_j')|| > \sqrt{\frac{3|A|\ln (2m)}{t}}$.
	Using Chernoff's bound,
	the probability that a subset  $A \in \mathcal{F}$ is bad  with respect to a random subset $A_j' \in \mathcal{F}'$ is
	\begin{align*}
	P\left[||A \cap A_j'|-|A \cap ([n]\setminus A_j')|| > \sqrt{\frac{3|A|\ln (2m)}{t}}\right] \leq 
	2e^{-\frac{3|A|\ln (2m)}{2t|A|}}=2(\frac{1}{2m})^\frac{3}{2t}.
	\end{align*}
	Any subset $A$ is bad  with respect to $\mathcal{F}'$ if $||A \cap A_j'|-|A \cap ([n]\setminus A_j')|| > \sqrt{\frac{3|A|\ln (2m)}{t}}$, for all $A_j' \in \mathcal{F}'$.
	So, $A$ is bad  with respect to $\mathcal{F}'$  with probability at most $2^t(\frac{1}{2m})^\frac{3t}{2t}=\frac{2^{t-1.5}}{m^{1.5}}$.
	Using union bound,
	the probability that some subset in $\mathcal{F}$ is bad  with respect to $\mathcal{F}'$ is at most 
	$m\frac{2^{t-1.5}}{m^{1.5}}$.
	So, if $2^t \leq \sqrt{m}$ (i.e., $t \leq \frac{1}{2}\log m$),
	the probability that any subset in $\mathcal{F}$ is bad  with respect to $\mathcal{F}'$ is at most $\frac{1}{2\sqrt{2}}$.
	Since the failure probability is less than $\frac{1}{2}$, in expected two iterations,
	we can obtain a family $\mathcal{F}'$ of $t$ subsets
	such that for every $A \in \mathcal{F}$, there is an $A_j' \in \mathcal{F}'$ with $||A \cap A_j'|-|A \cap ([n]\setminus A_j')|| \leq  \sqrt{\frac{3n\ln (2m)}{t}}$.
	\qed
\end{proof} 

Note that if $i \geq \sqrt{4.2n+1}$ and $|\mathcal{F}| = O(n^c)$, $c \in \mathbb{N}$,
a $D$-secting family for $\mathcal{F}$ of cardinality $O(\log n)$ can be computed as discussed above.
Note that this yields $D$-secting families 
of size much smaller than that guaranteed by Corollary \ref{cor:sqrtn}
for $\mathcal{F}$ provided $|\mathcal{F}|$ is polynomial in $n$.

\section{Bounds for $\beta_i(n)$}

In Section \ref{sec:Knn}, we established tight bounds for $\beta_D(n)$ when $D=[\pm i]$.
In this section, we study $\beta_D(n)$, when $D$ is a singleton set, i.e., $D=\{i\}$.

\subsection{Tight bounds for $\beta_1(n)$}
\label{subsec:1dist}
\begin{theorem}\label{thm:Knn21}
	$\beta_1(n) = \lceil\frac{n}{2}\rceil$, $n \in \mathbb{N}$.
\end{theorem}

\begin{proof}
	As mentioned in Section \ref{sec:intro},  when $D=\{1\}$, the family $\mathcal{F}$ 
	should consist of all the odd subsets of $[n]$.
	Let $\mathcal{R}$ be a minimum sized set of $\{-1,+1\}^n$ vectors such that
	for every odd set $A_o \in \mathcal{F}$, there exists a vector $R \in \mathcal{R}$ such that
	$\left\langle A_o, R \right\rangle-1 =0$.
	Consider the polynomial $M$ on $X=(x_1,\ldots,x_n)$.
	\begin{align}\label{eq:poly3}
	M(X)= \prod_{R \in \mathcal{R}} (\left\langle X,R \right\rangle -1)^2
	\end{align}
	Note that if $N'(Y)$ is obtained from $M(X)$ after domain conversion and multilinearization, $N'$ weakly represents the parity function.
	Using Lemma \ref{lemma:weak},
	$deg(M(X))=2|\mathcal{R}| \geq deg(N'(Y)) \geq n$ and therefore $|\mathcal{R}| \geq \lceil\frac{n}{2}\rceil$.
	In what follows, we demonstrate a construction of a family $\mathcal{F}'$ of cardinality $\lceil\frac{n}{2}\rceil$ such that for every odd subset $A \in \mathcal{F}$, there exists some $A' \in \mathcal{F}'$ with $|A \cap A'|-|A \cap ([n]\setminus A')|=1$.
	
	Consider the family $\mathcal{F}$
	consisting of all the odd subsets of $[n]$.
	Consider the case when $n$ is even; the odd case is similar except the ceilings in the final expression.
	Note that if $n \leq 2$, we can choose $\mathcal{F}'=\{\{1,2\}\}$ to get the desired intersection property.
	So, we consider the case when $n \geq 4$.
	Let $B_1=\{1,2,\ldots,\frac{n}{2}+1\}$.
	$B_2$ is obtained from $B_1$ by swapping $\{\frac{n}{2}+1\}$ with $\{\frac{n}{2}+2\}$, i.e.,  $B_2=\{1,2,\ldots,\frac{n}{2},\frac{n}{2}+2\}$.
	In general, $B_{j+1}$ is obtained from $B_{j}$ by replacing the point $\frac{n}{2}-j+2$ with 
	$\frac{n}{2}+j+1$. 
	We stop the process at $B_{\frac{n}{2}}=\{1,2,n,n-1,\ldots,\frac{n}{2}+2\}$. 
	Let $\mathcal{F}'=\{B_1,\ldots,B_\frac{n}{2}\} $.
	\begin{claimm}\label{claim:2}
		(i) For any odd subset $A_o \subseteq \{3,\ldots,n\}$, there exists some  $B_j$ and $B_l$ in $\mathcal{F}'$ such that
		$|A\cap B_j|=\lceil \frac{|A|}{2}\rceil$, and $|A\cap B_l|=\lfloor \frac{|A|}{2}\rfloor$, and
		(ii) For any even subset $A_e \subseteq \{3,\ldots,n\}$, there exists some  $B_j$ in $\mathcal{F}'$ such that
		$|A\cap B_j|= \frac{|A|}{2}$.
	\end{claimm}
		To see the correctness of the claim,
		consider an arbitrary set $A$, $A \subseteq \{3,\ldots,n\}$, such that  $|A \cap B_1| -  |A \cap ([n] \setminus B_1)|=d$, for some $d \in \mathbb{N}\setminus {0}$.
		Then, it follows from the construction that $|A \cap B_{\frac{n}{2}}| -  |A \cap ([n] \setminus B_{\frac{n}{2}})|=-d$.
		Observe that for any $j$, $1 \leq j \leq \frac{n}{2}-1$, the difference between 
		$|A \cap B_{j+1}| -  |A \cap ([n] \setminus B_{j+1})|$
		and $|A \cap B_{j}| -  |A \cap ([n] \setminus B_{j})|$ is either -2, 0 or 2.
		So, the claim follows.

	Now, to complete the proof, we need to consider the following exhaustive case for an odd subset $A_o$.
	\begin{enumerate}
		\item $A_o \subseteq \{3,\ldots,n\}$: $A_o$ has the desired intersection property using Claim \ref{claim:2}.
		
		\item $|A_o \cap \{3,\ldots,n\}|= |A_o|-1$: Using Claim \ref{claim:2}, 
		there exists some  $B_j$ in $\mathcal{F}'$ such that
		the even subset $A_o \cap \{3,\ldots,n\}$ is bisected by $B_j$.
		Clearly, $|A_o\cap B_j|=\lceil \frac{|A_o|}{2}\rceil$.
		
		\item $|A_o \cap \{3,\ldots,n\}|= |A_o|-2$: In this case, $ \{1,2\} \subset A_o$.
		From Claim \ref{claim:2}, 
		there exists some  $B_j$ in $\mathcal{F}'$ such that
		$|A_o' \cap B_j|= \lfloor \frac{|A_o'|}{2}\rfloor$, where  $A_o'=A_o \cap \{3,\ldots,n\}$.
		Then, $|A_o\cap B_j|=\lceil \frac{|A_o|}{2}\rceil$.
	\end{enumerate}
	
	This establishes that  $\beta_1(n)$ is at most  $\lceil \frac{n}{2} \rceil$ and completes the proof of Theorem \ref{thm:Knn21}.
	\qed
\end{proof}

\subsection{Bounds for $\beta_i(n)$, $i\geq 2$}
\label{subsec:idist}
In the following section, we extend the notion of $\beta_1(n)$ to arbitrary values of $i$.
Note that when $i=0$, $\beta_0(n)=\beta_{[\pm 1]}(n)= \lceil \frac{n}{2} \rceil$ (see Theorem \ref{thm:Knn}).
The case when $i=1$ is resolved by Theorem \ref{thm:Knn21}.
We assume that $i \geq 2$ in the remainder of the section.

\subsubsection{Proof of Theorem \ref{thm:iknn}}
\begin{statement}
	$\frac{n-i+1}{2}\leq \beta_i(n) \leq n-i+1$, $n \in \mathbb{N}$, $i \in [n]$.
\end{statement}

\begin{proof}
	
	Let $\mathcal{F}$ consist of all subsets of $[n]$ such that $A \in \mathcal{F}$ if and only if $|A| \cong i \mod 2$
	and $|A| \geq i$.
	Let $\mathcal{F}'=\{B_1=[i],B_2=B_1 \cup \{i+1\},\ldots,B_{n-i+1}=B_{n-i} \cup \{n\}\}$.
	Observe that $\mathcal{F}'$ is indeed an $i$-secting family for $\mathcal{F}$.
	Therefore, $\beta_i(n) \leq n-i+1$.
	In what follows, we prove the lower bound for $\beta_i(n)$
	assuming $i$ to be an  even integer greater than 1. The case for odd $i$ can be treated analogously.
	
	We invoke the notion of weak representation of the parity function to establish a lower bound.
	Let $\mathcal{F}$ denote the $2^n-1$ non-empty subsets of $[n]$.
	Let $\mathcal{F}'$ be a minimum cardinality $[\pm i]$-secting family for $\mathcal{F}$.
	Let $\mathcal{R}$ be the set of incidence vectors of sets in $\mathcal{F}'$, where each vector $R$ in $\mathcal{R}$ is
	an element of $\{-1,+1\}^n$.
	So, for any even subset $A_e \subseteq [n]$ with $|A_e| \geq i$, there exists a vector $R \in \mathcal{R}$ such that
	$\left\langle X_{A_e}, R \right\rangle-i =0$, where $X_{A_e}$ is the  0-1 incidence vector of $A_e$.
	We define the polynomials $P$, $M$ and $F$ on $X=(x_1,\ldots,x_n)$ as follows.
	\begin{align}
	M(X)=& \prod_{R \in \mathcal{R}} (\left\langle X,R \right\rangle -i)^2. \label{eq:mx}\\
	F(X)=& \sum_{S \in \binom{[n]}{i-1}}  \prod_{j \in S} x_j. \nonumber\\
	P(X)=&M(X)F(X).
	\end{align}
	Observe that (i) $P(X)$ evaluates to zero when $X=X_{A}$, for all subsets $A$ of size at most $i-2$ (since $F(X)$ vanishes for these subsets),
	(ii) $P(X)$ evaluates to zero when $X=X_{A_e}$, for all even subsets $A_e$ of size at least $i$ (since $M(X)$ vanishes for these subsets), and,
	(iii) $P(X)$ is strictly positive when $X=X_{A_o}$, for all odd subsets $A_o$ of size at least $i-1$.
	Consider the polynomial $Q$ on $Y=(y_1,\ldots,y_n)$, where each $y_j \in [\pm 1]$.
	\begin{align}\label{eq:poly5}
	Q\left(y_1,\ldots,y_n \right)=-P\left(x_1,\ldots,x_n\right)
	\end{align}
	where $x_j=\frac{1-y_j}{2}$, $1 \leq j \leq n$.
	Let $Q'(Y)$ be the multilinear polynomial obtained from $Q(Y)$ by replacing each occurrence of a $y_j^2$ by 1, repeatedly.
	Note that 
	(i) $Q'(Y)$ evaluates to zero for even subsets of $[n]$, and 
	(ii) if $Q'(Y)$ is non-zero on some odd subset $Y$, then $sign(Q'(Y))= sign(parity(Y))$.
	Therefore, $Q'(Y)$ weakly represents the parity function.
	From Lemma \ref{lemma:weak}, $Q'(Y)$ has degree at least $n$, 
	and $deg(P(X))=(i-1)+2|\mathcal{R}| \geq deg(Q'(Y)) \geq n$. So, $|\mathcal{R}| \geq \frac{n-i+1}{2}$.
	\qed

\end{proof}

\section{Bisecting $k$-uniform families}
\label{sec:k-uniform}

In this section, we discuss the problem of bisection for $k$-uniform families.
We focus on establishing bounds for $\beta_D(n,k)$ when $D=[\pm 1]$.

\subsection{Some observations for $\beta_{[\pm 1]}(n,k)$}
\begin{observation}\label{obs:2} 
	Let $n$ be an even integer and $ \mathcal{F}'$ be an optimal bisecting family for a family $\mathcal{F}=\binom{[n]}{k}$ 
	such that each subset $A' \in \mathcal{F}'$ has cardinality $\frac{n}{2}$. Then, $\beta_{[\pm 1]}(n,n-k) \leq \beta_{[\pm 1]}(n,k)$ 
\end{observation}

\begin{proof}
	It is not hard to see that the bisecting family $\mathcal{F}'$ for $\mathcal{F}$ is also a bisecting family for 
	$\overline{\mathcal{F}}=\binom{[n]}{n-k}$ when $n$ is even and each subset in $\mathcal{F}'$ is a part of an equal-sized bipartition of $n$. 
	\qed
\end{proof}

From Corollary \ref{prop:graph}, we know that $\beta_{[\pm 1]}(n,2)=\lceil \log n \rceil$.
Moreover, when $n$ is of the form $2^t$, for some $t  \in \mathbb{N}$,
we can obtain a bisecting family
$\mathcal{F}'=\{A_1,\ldots,A_{\log n}\}$ for the family $\mathcal{F}=\binom{[n]}{2}$
in the following way.
(i) For $j \in [n]$, obtain the $\log n$ bit binary code equivalent to $j-1$ and assign it to $j$.
(ii) Elements with $l$-th bit as 1 form the set $A_l$. 
Using Corollary \ref{prop:graph}, $\mathcal{F}'$ is an optimal bisecting family for $\mathcal{F}$, and 
$|A_l|=\frac{n}{2}$, for all $A_l \in \mathcal{F}'$. 
Using Observation \ref{obs:2}, it follows that $\beta_{[\pm 1]}(n,n-2) \leq  \log n$, when $n$ is a power of 2. 
However, when the difference between $n$ and $k$ is a small constant, we can achieve much better bounds for $\beta_{[\pm 1]}(n,k)$ as follows.

\subsection*{Proof of Theorem \ref{thm:constup}}
\begin{statement}
Let $\mathcal{F}=\binom{[n]}{k} \cup \binom{[n]}{k+1}\ldots \cup \binom{[n]}{n}$.
Then, $ \frac{n-k+1}{2} \leq \beta_{[\pm 1]}(\mathcal{F}) \leq \min \{\frac{n}{2}, n-k+1 \}.$
\end{statement}
\begin{proof}
The upper bound of $\frac{n}{2}$ follows from Lemma \ref{lemma:Knn1}.
Let $x=n-k$.
We obtain a bisecting family for $\mathcal{F}$ of cardinality $x+1$ in the following way.
Let $S$ and $T$ denote two disjoint $\lceil \frac{k}{2} \rceil $ and $\lfloor \frac{k}{2} \rfloor$ elements subset of $[n]$,
respectively. Let $c_1,\ldots,c_x$ denote the remaining 
elements of $[n]$.
Let $S_0=S$, and for any $j \in [x]$, $S_j=S_{j-1} \cup \{c_j\}$. Let 
$\mathcal{F}'=\{S_0,\ldots,S_x\}$.
We claim that
$\mathcal{F}'$ is a bisecting family for a $\mathcal{F}$.
For any set $A$ of cardinality $k'$, $k \leq k' \leq n$, that is not bisected by $S_0$, $|A \cap S_0| < \frac{k'}{2}$ and $|A \cap S_x| \geq \frac{k'}{2}$. The upper bound follows from the observation that $|A \cap S_{j+1}|$ differs from $|A \cap S_j|$ by at most 1.

The proof of the lower bound  $ \frac{n-k+1}{2}$ for $\beta_{[\pm 1]}(\mathcal{F})$ is 
in the same spirit  as the proof of the lower bound of Theorem \ref{thm:iknn}; we give the proof for completeness.
We assume that $k \geq 2$ and is even; the case when $k$ is odd is analogous.
Let $\mathcal{F}'$ be a minimum cardinality $[\pm 1]$-secting family for $\mathcal{F}$.
Let $\mathcal{R}$ be the set of incidence vectors of sets in $\mathcal{F}'$, where each vector $R$ in $\mathcal{R}$ is
an element of $\{-1,+1\}^n$.
We define the polynomials $P$, $M$ and $F$ on $X=(x_1,\ldots,x_n)$ as follows.
\begin{align}
M(X)=& \prod_{R \in \mathcal{R}} (\left\langle X,R \right\rangle)^2 \text{ (note the difference from Equation \ref{eq:mx})}.\\
F(X)=& \sum_{S \in \binom{[n]}{k-1}}  \prod_{j \in S} x_j.\\
P(X)=&M(X)F(X).
\end{align}
Observe that (i) $P(X)$ evaluates to zero when $X=X_{A}$, for all subsets $A$ of size at most $k-2$ (since $F(X)$ vanishes for these subsets),
(ii) $P(X)$ evaluates to zero when $X=X_{A_e}$, for all even subsets $A_e$ of size at least $k$ (since $M(X)$ vanishes for these subsets), and,
(iii) $P(X)$ is strictly positive when $X=X_{A_o}$, for all odd subsets $A_o$ of size at least $k-1$.
Note that if $Q'(Y)$ is obtained from $P(X)$ after domain conversion and multilinearization, $Q'(Y)$ weakly represents the parity function.
From Lemma \ref{lemma:weak}, $Q'(Y)$ has degree at least $n$, 
and $deg(P(X))=(k-1)+2|\mathcal{R}| \geq deg(Q'(Y)) \geq n$. So, $|\mathcal{R}| \geq \frac{n-k+1}{2}$.
\qed
\end{proof}

Note that using Theorem \ref{thm:constup} for $k=n-2$, we get, $\beta_{[\pm 1]}(n,n-2) \leq 3$.
This is surprising since 
(i) $\mathcal{F}=\binom{[n]}{n-2}$ has the same number of subsets as $\overline{\mathcal{F}}=\binom{[n]}{2}$,
(ii) the maximum number of sets of $\mathcal{F}$ and $\overline{\mathcal{F}}$ that can be bisected by a single set $A' \in \mathcal{F}'$
is $(\frac{n}{2})^2$, and
(iii) $\beta_{0}(n,2)=\lceil \log n \rceil$.

\begin{proposition}
	\label{prop:n2}
	$\beta_{[\pm 1]}(n,n-2) = 3$, for every even integer $n$ greater than 4.
\end{proposition}
\begin{proof}
	We only need to show that $\beta_{[\pm 1]}(n,n-2) > 2$.
	Note that since the hyperedges are of cardinality $n-2$, every set in  an optimal bisecting family $\mathcal{F}'$ is of cardinality
	$\frac{n}{2}-1$, $\frac{n}{2}$, or $\frac{n}{2}+1$.
	Consider an optimal bisecting family $\mathcal{F}'=\{A_1,A_2\}$ of cardinality 2 for $\mathcal{F}=\binom{[n]}{n-2}$. 
	Since $\beta_{[\pm 1]}(n,n-2) \leq 3$, any optimal bisecting family $\mathcal{F}'$ for $\mathcal{F}$ must contain at least one set of size other than $\frac{n}{2}$. Otherwise, using Observation \ref{obs:2}, $\mathcal{F}'$ is a bisecting family of  cardinality less than $\log n$ for $\binom{[n]}{2}$, a contradiction to Corollary \ref{prop:graph}.
	Without loss of generality, 
	assume that $|A_1| \neq \frac{n}{2}$.
	Using Observation \ref{obs:1}, we can also assume that  $|A_1| = \frac{n}{2}-1$.
	The rest of the proof is an exhaustive case analysis based on the cardinality of $A_2$.
	Let $ A^1 = A_1 \cap A_2$ and $A^2=A_1 \setminus A_2$.
	\begin{enumerate}
		\item $|A_2|=\frac{n}{2}$.  At least one of $A^1$ or $A^2$ is of size at least 2.
		The $(n-2)$-sized subset missing 2 elements of $[n]$ both from either $A^1$ or $A^2$ is not bisected by $\mathcal{F}'$.
		\item $|A_2|=\frac{n}{2}+1$.  If $|A^2| \geq 2$, 	the $(n-2)$-sized subset missing 2 elements both from $A^2$ is not bisected by $\mathcal{F}'$. So, $|A^2| \leq  1$.  If $A^2=\{y\}$,  then an $(n-2)$-sized subset missing $y$ and one element from $A^1$ is not bisected by $\mathcal{F}'$. If $A^2= \emptyset$, then any $(n-2)$-sized subset missing one element each from $A_1$ and $[n] \setminus A_2$ is not bisected by $\mathcal{F}'$.
		\item $|A_2|=\frac{n}{2}-1$. Using Observation \ref{obs:1}, this case is identical to Case 2.
	\end{enumerate} 
	\qed 
\end{proof}


\subsection{Proof of Theorem \ref{thm:klower}}

Note that the lower bound of  $\Omega(\sqrt{\frac{k (n-k)}{n}})$ for $\beta_{[\pm 1]}(n,k)$ is given by  Observation \ref{obs:ineq:2}.
However, when $k$ is a constant, Observation \ref{obs:ineq:2} asserts only an $\Omega(\sqrt{k})$ lower bound on $\beta_{[\pm 1]}(n,k)$.
An improved lower bound on $\beta_{[\pm 1]}(n,k)$ for constant $k$ given by Theorem \ref{thm:klower} is proven below. 

\begin{statement}
	\begin{align*}
	\beta_{[\pm 1]}(n,k) & \geq \begin{cases}
	\log (n-k+2)\text{, when $k$ is even and $\frac{k}{2}$  is odd,}\\
	\lceil (\log \lceil \frac{n}{\lceil\frac{k}{2}\rceil} \rceil)\rceil \text{, for any $k \geq 2$}.
	\end{cases}
	\end{align*}
\end{statement}

\begin{proof}
	We prove the first lower bound given in Theorem \ref{thm:klower} under the assumption that $k$ is even and $\frac{k}{2}$  is odd.
	Let $\mathcal{F}'=\{A_1',\ldots,A_t'\}$ be a bisecting family for the family $\mathcal{F}=\binom{[n]}{k}$.
	For every $A_j' \in \mathcal{F}'$, let $\mathcal{F}_j$ be the collection of $k$-sized sets that are bisected by $A_j'$.
	We estimate a lower bound for $t$.
	We associate a graph $G(\mathcal{F})$ with the collection $\mathcal{F}$ of $k$-sized sets 	in the following way:
	\begin{align*}
	V(G(\mathcal{F}))&=\{ S \in \binom{[n]}{\frac{k}{2}}:S \subseteq A\text{, }A \in \mathcal{F}\} \\
	E(G(\mathcal{F}))&=\{\{S_1,S_2\}: S_1 \cap S_2 = \emptyset, S_1,S_2 \in V(G(\mathcal{F}))\}.
	\end{align*}
	Observe that $G(\mathcal{F})$ is the Kneser graph $KG(n,\frac{k}{2})$ (for definitions and results related to Kneser graphs, see \cite{bollobas2004extremal,aigner2010proofs}).
	For every $k$-sized subset $A \in \mathcal{F}$,
	there are $\binom{k}{\frac{k}{2}}$ edges in $E(G(\mathcal{F}))$: an edge between any two disjoint 
	$\frac{k}{2}$ sets. 
	From the definition of $\mathcal{F}_1,\ldots, \mathcal{F}_t$, $\displaystyle\cup_{j=1}^t G(\mathcal{F}_j)=G(\mathcal{F})$.
	
	\begin{claimm}
		Each $G(\mathcal{F}_j)$ is a bipartite graph.
	\end{claimm}
	Let $A \in \mathcal{F}_j$. Consider a fixed $\frac{k}{2}$ sized subset $S$ of $A$.
	If $|S \cap A_j'| > \lfloor \frac{k}{4}\rfloor$, $S$ is placed in the first partite set of $G(\mathcal{F}_j)$;
	otherwise $S$ is placed in the second partite set of $G(\mathcal{F}_j)$.
	Note that since $\frac{k}{2}$ is odd, $|S \cap A_j'|$ can never be equal to   $|S \cap ([n] \setminus A_j')|$.
	It is now easy to see that there is no edge inside the first or second partite set of $G(\mathcal{F}_j)$.
	
	$G(\mathcal{F}_1),\ldots,G(\mathcal{F}_t)$ are bipartite graphs whose union covers $G(\mathcal{F})$. Since  $G(\mathcal{F})$ is
	is the Kneser graph $KG(n,\frac{k}{2})$, its chromatic number is $n-k+2$ (see \cite{LOVASZ1978319,aigner2010proofs}). So, using Proposition  \ref{lemma:bc}, we get, $t \geq \lceil \log (n-k+2) \rceil$
	\footnote{Note that Proposition \ref{lemma:bc} does not guarantee equality since the  $\lceil \log (n-k+2) \rceil$ bipartite graphs
		that cover $G(\mathcal{F})$ as per Proposition \ref{lemma:bc} may not correspond to valid $\mathcal{F}_j$'s.}.
	That is, $\beta_{[\pm 1]}(n,k) \geq \lceil \log (n-k+2) \rceil$, when $k$ is even and $\frac{k}{2}$ is odd. This concludes the proof of the first lower bound given by Theorem \ref{thm:klower}.
	
	To prove the second lower bound of Theorem \ref{thm:klower},
	consider a bisecting family $\mathcal{F}'=\{A_1',\ldots, A_t'\}$ of $\mathcal{F}=\binom{[n]}{k}$.
	Observe that for every $\lceil\frac{k}{2}\rceil+1$-sized set $S \subseteq [n]$, there exists an $A_j' \in \mathcal{F}'$ such that
	$S \cap A_j' \neq \emptyset$ and $S \cap ([n]\setminus A_j') \neq \emptyset$. 
	For every $A_j' \in \mathcal{F}'$, let $\mathcal{F}_j$ be the collection of $\lceil\frac{k}{2}\rceil+1$-sized sets 
	that has a non-empty intersection with both $A_j'$ and $[n] \setminus A_j'$.
	Observe that 
	\begin{align} \label{equal:1}
	\bigcup\limits_{j=1}^{t} \mathcal{F}_j=\binom{[n]}{\lceil\frac{k}{2}\rceil+1}.
	\end{align}
	Construct hypergraphs $G_1,\ldots, G_{t}$, where $V(G_j)=[n]$ and $E(G_j)=\mathcal{F}_j$.
	To each point $v \in [n]$, assign an ${t}$ length 0-1 bit vector: $j$th bit is 1 if and only if $v \in A_j$.
	Color the points in $[n]$ with the decimal equivalent of its bit vector. 
	Let $f:[n] \rightarrow \{0,1,\ldots,2^{t}-1\}$ denote this coloring.
	We show that none of the $\binom{[n]}{\lceil\frac{k}{2}\rceil+1}$
	sets remain monochromatic under $f$.
	Assume for the sake of contradiction that $S \in \binom{[n]}{\lceil\frac{k}{2}\rceil+1}$
	is monochromatic under $f$. From Equation \ref{equal:1}, there exists an $\mathcal{F}_j$ such that 
	$S \in \mathcal{F}_j$.  From the definition of $\mathcal{F}_j$, $S$ has non-empty intersection with both $A_j'$ and $[n] \setminus A_j'$.
	Therefore, the $j$th bits of the ${t}$ length 0-1 bit vectors of all the points in $S$ cannot be the same.
	Therefore, $S$ contains at least two points of different color under $f$, i.e., $S$ is not monochromatic.
	It is well known that the chromatic number of $\binom{[n]}{\lceil\frac{k}{2}\rceil+1}$, $\chi(\binom{[n]}{\lceil\frac{k}{2}\rceil+1})$, is $\lceil \frac{n}{\lceil\frac{k}{2}\rceil} \rceil$.
	Since $f$ 
	uses $2^{t}$ colors,
	we have, $2^t \geq \lceil \frac{n}{\lceil\frac{k}{2}\rceil} \rceil $
	Therefore, $\beta_{[\pm 1]}(n,k) = |\mathcal{F}'| =t \geq \lceil (\log \lceil \frac{n}{\lceil\frac{k}{2}\rceil} \rceil)\rceil$.
	
	This completes the proof of Theorem \ref{thm:klower}.
	\qed
\end{proof}

\subsection{Proof of Theorem \ref{thm:n_k_linear}}

We know that $\beta_{[\pm 1]}(n)=\lceil \frac{n}{2}\rceil$ (see Theorem \ref{thm:Knn}).
The  number of $\frac{n}{2}$-sized subsets of $[n]$ that can be bisected by a single subset $A' \subseteq [n]$ is at most $2(\binom{\frac{n}{2}}{\frac{n}{4}})^2$.
This gives a trivial lower bound of $\Omega(\sqrt n)$ for $\beta_{[\pm 1]}(n,\frac{n}{2})$.
In this section, we prove a stronger result using a theorem of Keevash and Long \cite{keevash2017frankl} which is an improvement over a theorem of Frankl and R\"{o}dl \cite{MR871675}. 
Given $q \in \mathbb{N}$, a set $\mathcal{C}$ is called a $q$-ary code if $\mathcal{C} \subseteq [q]^n$, for $q \geq 2$.
For any $x, y \in [q]^n$, the \emph{Hamming distance} between $x$ and $y$, where $x=(x_1,\ldots,x_n)$ and $y=(y_1,\ldots,y_n)$, denoted by $d_H(x,y)$, is
$|\{i \in [n]:x_i \neq y_i|$. For any code $\mathcal{C}$, let $d(\mathcal{C})$ be the set of 
all the Hamming distances allowed for any $x,y \in \mathcal{C}$.
A code is called $d$-\emph{avoiding} if $d \not\in d(\mathcal{C})$.
We have the following upper bound on the cardinality of a $d$-\emph{avoiding} code $\mathcal{C}$ as given in 
\cite{keevash2017frankl}.

\begin{theorem}\cite{keevash2017frankl}\label{thm:keevesh}
	Let $\mathcal{C} \subseteq [q]^n$ and let $\epsilon$ satisfy $0 < \epsilon < \frac{1}{2}$.
	Suppose that $\epsilon n < d < (1-\epsilon)n$ and $d$ is even if $q =2$. If $d \not\in d(\mathcal{C})$,
	then $|\mathcal{C}| \leq q^{(1-\delta)n}$, for some positive constant $\delta=\delta(\epsilon)$.
\end{theorem}

In what follows, we prove Theorem \ref{thm:n_k_linear}. The proof is similar to the proof by Frankl-Rodl \cite{frankl1987} in resolution of the Galvin's Problem.

\begin{statement}
	Let $c$ be a constant such that $0 < c < \frac{1}{2}$ and $n \in \mathbb{N}$. If $cn < k < (1-c)n$, then
	\begin{align*}
	\max\Big\{\beta_{[\pm 1]}(n,k),\beta_{[\pm 1]}(n,k-1),\beta_{[\pm 1]}(n,k-2),\beta_{[\pm 1]}(n,k-3)\Big\}	\geq \delta n, 	
	\end{align*}
	where $\delta=\delta(c)$ is some real positive constant.
\end{statement}

\begin{proof}
	Consider a bisecting family $\mathcal{F}'=\{A'_1,\ldots,A'_m\}$ of minimum cardinality for $\binom{[n]}{l}$,
	where $cn < l < (1-c)n$ is even and $\frac{l}{2}$ is odd, for some constant $c$, $0 < c < \frac{1}{2}$.
	Let $X_A$ denote the 0-1 incidence vector corresponding to a set $A \subseteq [n]$.
	Let $V$ denote the vector space generated by the incidence vectors of $\mathcal{F}'$ over $\mathbb{F}_2$.
	Observe that for any $A \in \binom{[n]}{l}$, there exists an $A' \in\mathcal{F}'$ such that
	$|A \cap A'|=\frac{l}{2}$.
	Since $\frac{l}{2}$ is odd, $\left\langle X_A,X_{A'} \right\rangle=1$, i.e., $X_A \not\in V^{\perp}$,
	where $ V^{\perp}$ is the 
	subspace of the vector space $\{0,1\}^n$ over $\mathbb{F}_2$ which contains all the vectors perpendicular to $V$. 
	So, $V^{\perp}$ is a subspace containing no vector of weight $l$.
	For any $X_B,X_C \in V^{\perp}$, $X_B + X_C$ has weight $|B \triangle C| \neq l$. Moreover,
	$l$ is even.
	Since $cn < l < (1-c)n$, using Theorem \ref{thm:keevesh}, there exists an positive constant $\delta=\delta(c)$ such that $|V^{\perp}| \leq 2^{n(1-\delta)}$.
	So, $dim(V^{\perp}) \leq n - \lfloor\delta n\rfloor$.
	It follows that $dim(V) \geq \lfloor\delta n\rfloor$. 
	To complete the proof of the theorem, note that for any $k$, there exists an $l \in \{k,k-1,k-2,k-3\}$ such that $l$ is even and $\frac{l}{2}$ is odd.
	\qed
\end{proof}

\subsection{$\beta_{0}(n,k)$ and computation of bisecting families }
\label{subsec:lll}

An important probabilistic tool  used in this section is the Lov\'{a}sz local lemma \cite{loverd1975}.
Let $\mathcal{F}$ be a family of subsets of $[n]$.
The {\it dependency} of a set $A \in \mathcal{F}$
denoted by $d(A,\mathcal{F})$ 
is the number of subsets $\widehat{A} \in \mathcal{F}$, such that
(i) $|A \cap \widehat{A}| \geq 1$, and
(ii) $A \neq \widehat{A}$.   
The {\it dependency} of a family $\mathcal{F}$, denoted by $d(\mathcal{F})$ or simply $d$, is the 
maximum dependency of any subset $A$ in the family $\mathcal{F}$.
We have the following corollary of the Lov\'{a}sz local lemma from \cite{Moser:2010:CPG:1667053.1667060}.

\begin{lemma}\cite{Moser:2010:CPG:1667053.1667060} \label{cor:lll}
	Let $\mathcal{P}$ be a finite set of mutually independent random variables
	in a probability space. Let ${\cal A}$ be a 
	finite set of events determined by these variables,  where $m=|{\cal A}|$.
	For any $A \in {\cal A}$, let $\Gamma(A)$ denote the set of all the events in ${\cal A}$ that depend on $A$.
	Let $d = \max_{A \in {\cal A}}|\Gamma(A)|$.
	If 	$\forall A \in {\cal A} :P[A] \leq  p  \text{ and } ep(d+1) \leq 1$, 
	then
	an assignment of the variables not violating any of
	the events in ${\cal A}$ can be computed using expected 
	$\frac{1}{d}$ resamplings per event and
	expected $\frac{m}{d}$ resamplings in total.
\end{lemma}

\subsubsection*{Proof of Theorem \ref{thm:lll}}
\begin{statement}
	For a family $\mathcal{F}$ consisting of $k$-sized subsets of $[n]$ and dependency $d$,
	$\beta_{[\pm 1]}(\mathcal{F}) \leq \frac{\sqrt{k}}{c}(\ln(d+1)+1)$, where $c = 0.67$.
\end{statement}
\begin{proof}
	
	Let $\mathcal{F}$ be a family of $k$-sized subsets of $[n]$, $\mathcal{F} \subseteq \binom{[n]}{k}$, with dependency $d$.
	Assume that $k$ is even.
	Consider a family $\mathcal{F}'=\{A_1',\ldots,A_t'\}$:
	each $A_j' \in \mathcal{F}'$ is a random subset of $[n]$ where each point $x \in [n]$ is chosen into $A_j'$ 
	independently with probability $\frac{1}{2}$.
	Let $p$ be the probability that a fixed subset $A \in \mathcal{F}$ is bisected by some $A_j' \in \mathcal{F}'$.

	\begin{align*}
	p=\frac{\binom{k}{\frac{k}{2}}}{\binom{k}{0}+\binom{k}{1}+\ldots+\binom{k}{k}} 
	\geq  \frac{c}{\sqrt{k}} \text{, where } c=0.67.
	\end{align*}
	So, the failure probability that $A$ is not bisected by $A_j'$ is $1-p$ which is at most 
	$1-\frac{c}{\sqrt{k}}$.
	Therefore, the failure probability that $A$ is not bisected by any $A_j'\in \mathcal{F}'$ is $(1-p)^t$ which is at most $(1-\frac{c}{\sqrt{k}})^t \leq e^{-\frac{ct}{\sqrt{k}}}$.
	Using Lemma \ref{cor:lll}, we get $t \geq \frac{\sqrt{k}}{c}(\ln(d+1)+1)$.
	This implies that there exists a bisecting family for any family $\mathcal{F}$ of $k$-sized sets
	of size $\frac{\sqrt{k}}{c}(\ln(d+1)+1)$, where $d$ denotes the dependency of family $\mathcal{F}$.
	
	In fact, if $\mathcal{F}$ is $\binom{[n]}{k}$ and we choose the subsets $A_j' \in \mathcal{F}'$ of cardinality exactly $\frac{n}{2}$
	uniformly and independently at random
	from $\binom{[n]}{\frac{n}{2}}$, then $p = \frac{\binom{\frac{n}{2}}{\frac{k}{2}}^2}{\binom{n}{k}}\geq c_1\sqrt{\frac{n}{(n-k)k}}$ $(c_1 \geq 0.53)$.
	Therefore, the failure probability that $A$ is not bisected by any $A_j'\in \mathcal{F}'$ is $(1-p)^t$.
	Using Lemma \ref{cor:lll}, we can compute a
	bisecting family for $\binom{[n]}{k}$
	of size $\frac{1}{c_1} \sqrt{\frac{k(n-k)}{n}}(\ln(d+1)+1)$.
	Therefore, using Observation \ref{obs:ineq:2},
	$\beta_{[\pm 1]}(n,k)$ is  $O((\ln(d+1)+1))$-approximable.
	
	The proof for the case when $k$ is odd is similar to the above proof. In fact, we get a small constant factor improvement over
	the bound given in Theorem \ref{thm:lll}.
	\qed
\end{proof}
Let $m=|\mathcal{F}|$. Since, $d+1 \leq m \leq \binom{n}{k} < (\frac{en}{k})^k$,
we get, 
$\beta_{[\pm 1]}(n,k) \leq \frac{1}{c_1} \sqrt{\frac{k(n-k)}{n}}(\ln m+1) \leq \frac{k}{c_1} \sqrt{\frac{k(n-k)}{n}}\ln (\frac{en}{k})$.

\section{Discussion and open problems}
\label{sec:con}

The discrepancy interpretation of bisecting families leads us to the investigation of $\beta_{[\pm 1]}(\mathcal{F})$
for recursive Hardamard set systems.
\subsection*{Bisecting families for Hadamard set systems}
\begin{defn}
	A Hadamard matrix $H$ is a $n \times n$ matrix with 
	(i) each entry being either $+1$ or $-1$, and
	(ii) any two distinct columns being orthogonal, i.e., $H^TH=nI$, where $I$ is the $n \times n$ identity matrix.
\end{defn}
By convention, the first row and first column of $H$ are all ones.
By a recursive construction, $H(k)$ of size $2^k \times 2^k$ can be obtained from $H(k-1)$ of size $2^{k-1} \times 2^{k-1}$ as follows:
\begin{align*}
H(k)= \left[ \begin{array}{cc}
H(k-1) & H(k-1) \\
H(k-1)  & -H(k-1)  \\
\end{array} \right],
\end{align*}
where $H(0)=1$.
Note that except the first row, every other row of the Hadamard matrix $H(k)$ must contain equal number of
1's and -1's, since the columns are orthogonal and $H(k)$ is symmetric.
Let $A=\frac{1}{2}(H(k)+J(k))$, where
$J$ is the $2^k \times 2^k$ matrix whose every entry is $+1$.
The matrix $A$ corresponds to the Hadamard set system $HF(k)$,
where $HF(k)=\{A_1,\ldots,A_{2^k}\}$, and,
$j \in A_i$ if and only if the $(i,j)$ entry of $A$ is one.
So, from construction, every subset $A_j \in HF(k)$ except $A_1$ is of cardinality exactly $2^{k-1}$.
It is a well known fact that a Hadamard set system $HF$ of order $n \times n$ has a discrepancy at least 
$\frac{\sqrt{n-1}}{2}$ \cite[p. 106]{mat1999}. 
Therefore, $\beta_{[\pm 1]}(HF(k)) \geq 2$.
In what follows, we show that $\beta_{[\pm 1]}(HF(k)) \leq 2$ for all Hadamard set systems obtained from the recursively constructed Hadamard matrix $H(k)$, $k>1$.
Consider the Hadamard set system $HF(k)$, which is represented by the incidence matrix
$A$. Let $B_1=\{1,\ldots,2^{k-1}\}$.
Observe that $A_1$ through $A_{2^{k-1}}$ of $HF(k)$ are bisected by $B_1$ due to the recursive construction.
$A_{2^{k-1}+1}$ represented by the $2^{k-1}+1$th row of $A$ is not bisected by $B_1$.
In fact, $|A_{2^{k-1}+1} \cap B_1| - |A_{2^{k-1}+1}  \cap ([2^k] \setminus B_1)| = 2^{k-1}$. 
The subsets $A_{2^{k-1}+2}$ through $A_{2^{k}}$ of $HF(k)$ are bisected by $B_1$
since every row, except the first row, of $H(k-1)$  and  $-H(k-1)$
contain equal number of 1's and -1's.
$A_{2^{k-1}+1}$ represented by the $2^{k-1}+1$th row of $A$ can be  bisected by a second subset $B_2=\{1,\dots,2^{k-2}\}$.
So, this establishes $\beta_{[\pm 1]}(HF(k))=2$, $k>1$.

From the above discussion, it is clear that discrepancy of a set system $\mathcal{F}$ can be arbitrarily large as compared to
$\beta_{[\pm 1]}(\mathcal{F})$.
On the other extreme, we know that  discrepancy  of a family of 2-sized subsets $\mathcal{F}$ of $[n]$ cannot exceed 2, 
whereas $\beta_{[\pm 1]}(\mathcal{F})$ can be as large as $\log n$.
Thus, there exists families $\mathcal{F}$ and $\mathcal{G}$ where $\beta_{[\pm 1]}(\mathcal{F})$ and $disc(\mathcal{G})$ are
constants whereas $disc(\mathcal{F})$ and $\beta_{[\pm 1]}(\mathcal{G})$ are arbitrarily large.
However, this does not rule out a possible relationship between these two parameters and other hypergraph parameters.
One possibility of making progress in this direction is obtaining tight upper and lower bounds for $\beta_{[\pm 1]}(\mathcal{F})$.
Recall that the discrepancy of a family $\mathcal{F}$ is the minimum $i \in \mathbb{N}$ such that $\beta_{[\pm i]}(\mathcal{F})\leq 1$.
Below, we demonstrate the usage of such tight bounds where $\mathcal{F}=2^{[n]}$ and $n$ is a power of 2.
From Theorem \ref{thm:Knn}, we have, $\frac{n}{2} \geq \beta_{[\pm 1]}(n) \geq 2\beta_{[\pm 2]}(n) \geq \cdots \geq 2^j \beta_{[\pm 2^j]}(n)$.
So, when $j=\log (\frac{n}{2})$, we get, $\beta_{[\pm 2^j]}(n) \leq 1$. This gives a known trivial upper bound for $disc(\mathcal{F})$.

As mentioned in the introduction, $\beta_{[\pm 1]}(E)$ is $\lceil \log \chi(G)\rceil$ for a graph $G(V,E)$. 
We know that it  is impossible to approximate the chromatic number
of graphs on $n$ vertices within a factor of $n^{1-\epsilon}$ for any 
fixed $\epsilon > 0$, unless $NP \subseteq ZPP$ (see Feige and Killian \cite{Feige1998187}).
Therefore, it is not difficult to see that under the
assumption $NP \not\subseteq ZPP$, no polynomial time algorithm can approximate $\beta_{[\pm 1]}(E)$ for an $n$-vertex graph $G(V,E)$ within an additive approximation factor of $(1-\epsilon)\log n-1$, for any fixed $\epsilon> 0$.

In Section \ref{sec:prelim}, we have seen that $\beta_D(n,k)$ is not monotone with $k$ in general.
However, it is possible that $\beta_D(n,k)$ is  monotone with $k$ in certain ranges, say when $ k \leq \frac{n}{2}$.
In Section \ref{subsec:idist}, we established the lower bound of $\frac{n-i+1}{2}$  for $\beta_i(n)$. However, 
the best upper bound we have for this case is just $n-i+1$. So, there is a gap between the lower and upper bounds for
$\beta_i(n)$.

\section*{Acknowledgements}
The research of the third author is supported by the doctoral fellowship program of Ministry of Human Resources and Development, Govt. of India.

\end{document}